\documentclass{article}
\usepackage{amsthm,amsfonts, amsbsy, amssymb,amsmath,graphicx}
\usepackage{graphics}
\usepackage[english]{babel}

\newtheorem{theorem}{Theorem}
\newtheorem{example}{Example}
\newtheorem{remark}{Remark}
\newtheorem{definition}{Definition}

\date{}

\begin{document}

\markboth{M.V.Zenkina}
{The parity hierarchy and new invariants of knots in thickened surfaces}

\title{THE PARITY HIERARCHY AND NEW INVARIANTS OF KNOTS IN THICKENED SURFACES}

\author{M.V.ZENKINA}

\maketitle

\begin{abstract} In the present paper, we construct an invariant for virtual knots in the thickened
sphere $S_{g}$ with $g$ handles; this invariant is a Laurent polynomial in $2g+3$ variables. To this end, we use a modification of the Wirtinger presentation of the knot group and the concept of parity introduced by V.O.Manturov. By using this invariant, one can prove that the knots shown in Fig. \ref{ris1.12(1)} are not equivalent \cite{GMV1}. The section $4$ of the paper is devoted to an enhancement of the invariant (construction of the invariant module) by using the parity hierarchy concept suggested by V.O.Manturov.
Namely, we discriminate between odd crossings and two types of even crossings; the latter two
types depend on whether an even crossing remains even/odd after all odd crossings of
the diagram are removed. The construction of the invariant also works for virtual knots.
\end{abstract}

\section{Introduction}

In \cite{KaV}, L.H.Kauffman introduced virtual knot theory, which was an important generalization of classical knot theory.
Virtual knots are knots in thickened $2$-surfaces considered up to isotopy and stabilizations/destabilizations.
Many invariants of classical knots have been generalized to virtual knot theory, for example Khovanov homology and different modifications of the Alexander polynomial associated with
Wirtinger's presentation \cite{VA}. In the standard Wirtinger presentation of the knot group generators correspond to arcs
and relations correspond to crossings. In the present paper, we shall construct invariants for virtual knots by using the Wirtinger presentation of the knot group and the concept of parity
introduced by V.O.Manturov \cite{VOM}. We first shall use the parity hierarchy suggested by V.O.Manturov in order to construct new invariants.

The paper is organized as follows. In Section $2$ we shall construct a polynomial $s$, which is a generalization of the Alexander polynomial of knots in thickened spheres $S_g$ with $g$ handles. In order to construct this invariant, we shall use the Wirtinger presentation of the knot group and the concept of parity of crossings, introduced by V.O.Manturov \cite{VOM}. In Section $3$ we shall prove the invariance of polynomial $s$ and prove that the knots shown in Fig. \ref{ris1.121} are not equivalent. A question about the equivalence of these knots is stated in \cite{GMV1}. In Section $4$ of this paper we first use the parity hierarchy suggested by V.O.Manturov and construct an invariant module $N(K)$. We discriminate between odd crossings and two types of even crossings; the latter two
types depend on whether an even crossing remains even/odd after all odd crossings of
the diagram are removed. In Section $5$ we shall construct a simplification of $N(K)$ in order to obtain
an invariant polynomial for virtual knots.

\begin{figure}
\centering\includegraphics[width=400pt]{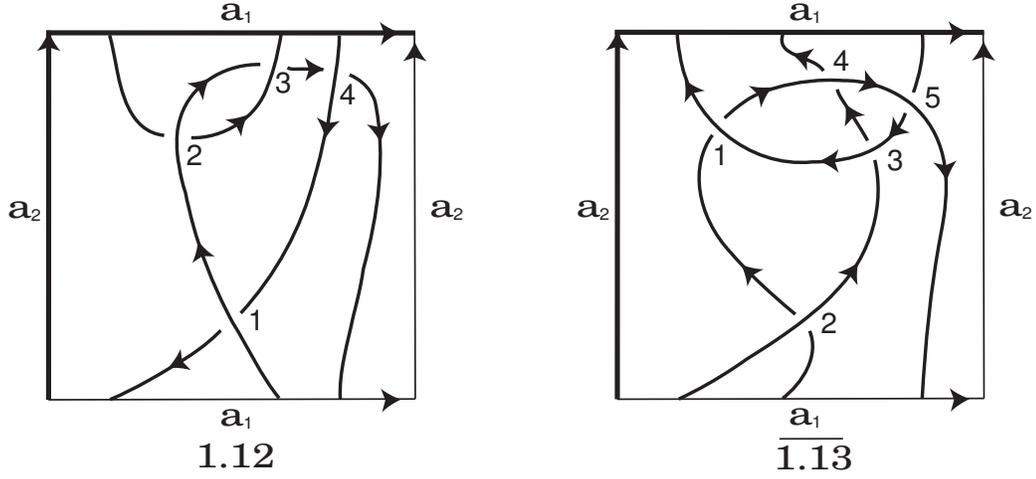}
\caption{Knot diagrams $1.12$ and $\overline{1.13}$}
\label{ris1.121}
\end{figure}

\section{The main construction}

Let $S_{g}$ be the closed $2$-surface of genus $g$. Let $L$ be a knot in the thickening $S_{g}\times I$ of $S_{g}$, $I$ being an interval.

\begin{definition}
Let $G$ be a graph with the set of vertices $V(G)$ and the set of edges $E(G)$. We think of an edge as an equivalence class of the two half-edges
constituting the edge. We say that a vertex $v \in V(G)$ has the degree $4$ if $v$
is incident to $4$ half-edges. A graph whose vertices have the same degree $4$
is called a \emph{$4$-graph}.
\end{definition}

Generically, a projection of $L$ to $S_{g}$ is a regular $4$-graph with each vertex having an over/under crossing structure. Two graphs are equivalent if and only if one of them is obtained from the other graph by using the standard Reidemeister moves, see Fig. \ref{move}. Let us represent $S_{g}$ as a $4g$-gon $P_{g}$ with opposite sides identified.  Later on, we shall depict the knot diagram on $P_{g}$. We fix an orientation on $S_g$.
Let $K$ be a diagram of an oriented knot $L$ on $S_g$ with $n$ crossings. Let $V(K)$ be the set of all crossings of the knot diagram $K$. We shall 	
construct a \emph{chord diagram} for any knot diagram on $P_{g}$ in the following way. A chord
diagram consists of a circle (with a fixed point, not a pre-image of a crossing) on which the pre-images of the overcrossing and the undercrossing (for each crossing) are connected by a chord (Fig. \ref{gaus}).

\begin{figure}
\centering\includegraphics[width=350pt]{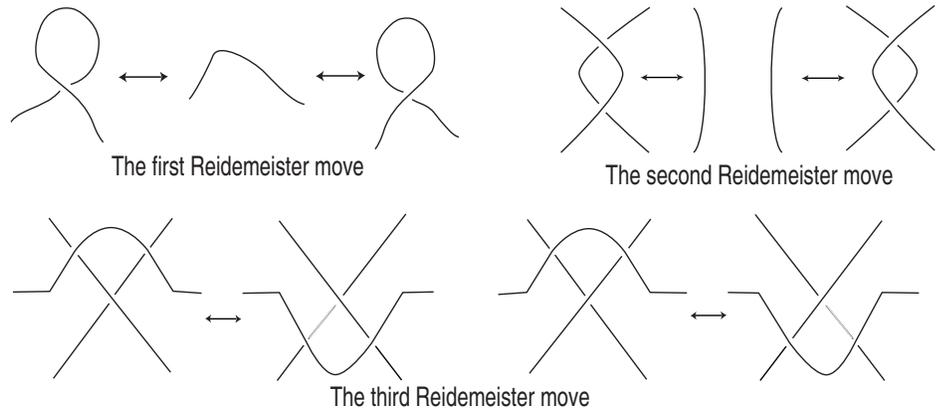}
\caption{Reidemeister moves} \label{move}
\end{figure}

\begin{figure}
\centering\includegraphics[width=300pt]{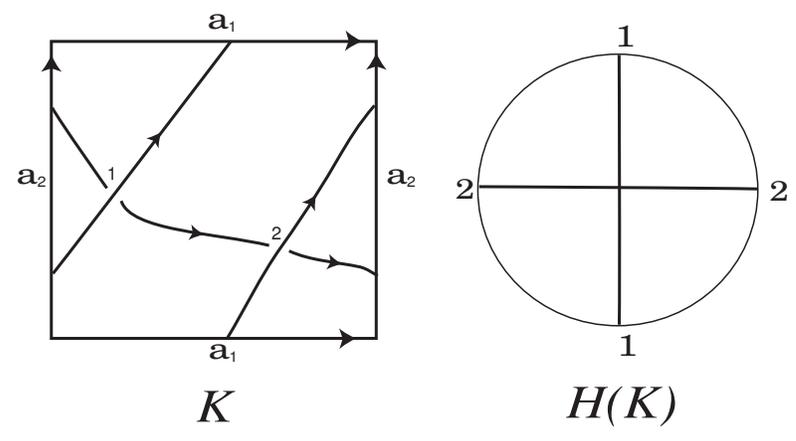}
\caption{Construction of the chord diagram} \label{gaus}
\end{figure}

\begin{definition}

Let $K_1$ and $K_2$ be knot diagrams obtained from each other by one of the Reidemeister moves, so that the number of crossings of $K_{2}$ is less than or equal to that of $K_{1}$. A parity is a rule which associates $0$ or $1$ with every crossing of each diagram in such a way that:

1. If $K_2$ is obtained from $K_1$ by a first Reidemeister move, then the crossing of the diagram $K_1$ in the first move is even;

2. If $K_2$ is obtained from $K_1$ by a second Reidemeister move, then the both crossings in the second move have the same parity;

3. If $K_2$ is obtained from $K_1$ by a third Reidemeister move, then we have a natural correspondence between triple of the crossings of $K_1$ and triple of the crossings of $K_2$ $(a_1$, $a_2)$, $(b_1$, $b_2)$, $(c_1$, $c_2)$ in the third move, see Fig. \ref{3move1}. We require that
a) the parity of the corresponding crossings coincides;
b) among the crossings $a_1, b_1, c_1$ the number of odd crossings is even i.e. $0$ or $2$.

\begin{figure}
\centering\includegraphics[width=350pt]{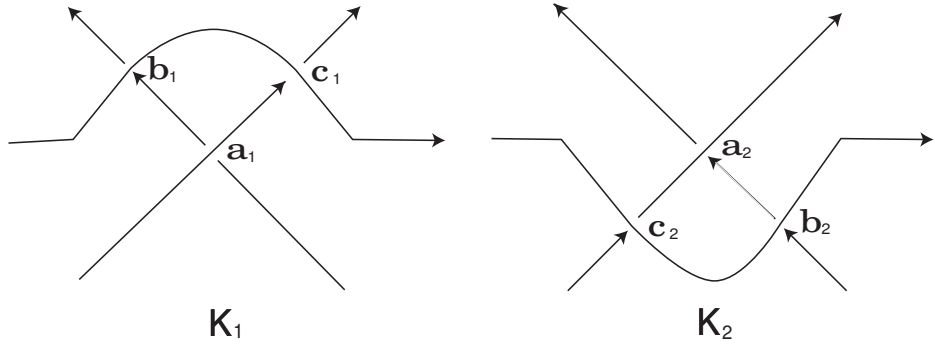}
\caption{Correspondence between the crossings of $K_1$ and $K_2$} \label{3move1}
\end{figure}

4. The parity of crossings not taking part in Reidemeister moves does not change after these moves is performed \cite{VOM}.

\end{definition}

\begin{definition}
By \emph{an arc} of a knot diagram we mean a connected
component of the diagram.
Thus, each arc always goes over; it starts and stops at undercrossings. For
knot diagrams in general position, each vertex is incident to three arcs.
\end{definition}

Let us define the Gaussian parity as follows. Let $L$ be a knot in $S_{g}\times I$. We are going to distinguish between even and odd crossings of a diagram $K$ of $L$. A crossing $v$ is \emph{even} if the number of the chords on the chord diagram $H(K)$ which intersect the chord, corresponding to the crossing $v$, is even. Otherwise, it is \emph{odd}. In even crossings we shall set local labels on arcs as shown in Fig. \ref{ris2}, and in odd crossings we shall set as shown in Fig. \ref{ris1}.

\begin{figure}
\centering\includegraphics[width=70pt]{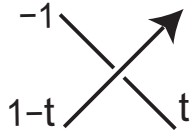}
\caption{Local labels in even crossings} \label{ris2}
\end{figure}

\begin{figure}
\centering\includegraphics[width=70pt]{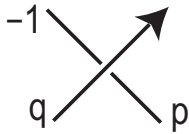}
\caption{Local labels in odd crossings} \label{ris1}
\end{figure}

Let us enumerate crossings of the diagram by integers from $1$ to $n$, where $n$ is the number of crossings of $K$. Evidently, the number of arcs is also equal to $n$. We enumerate the arcs of the diagram by the same numbers as crossings in such a way that the arc with number $k$ originates from the crossing with the number $k$. Denote by ${a_1,a_2,...,a_{2g}}$ the sides of $P_g$ (Fig. \ref{vos}).

\begin{figure}
\centering\includegraphics[width=150pt]{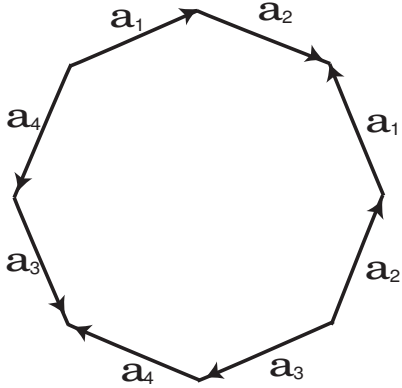}
\caption{Fundamental octagon} \label{vos}
\end{figure}

In the $4g$-gon $P_{g}$ we choose one side from each 	
pair of pasted sides. The sides divide some arcs into subarcs. First, let us define labels of subarcs as follows. We assign label $(0,0,...,0)$ to each subarc starting in $k$-th crossing. Then, when passing through a side $a_m$ the $m$-th element of the label is changed by $+1$	if we pass through a selected side $a_m$, i.e. $(0,0,...,1,...,0)$, or by $-1$ if we pass through a non-selected side $a_m$, i.e. $(0,0,...,-1,...,0)$ (see Fig. \ref{ris13}).

\begin{figure}
\centering\includegraphics[width=200pt]{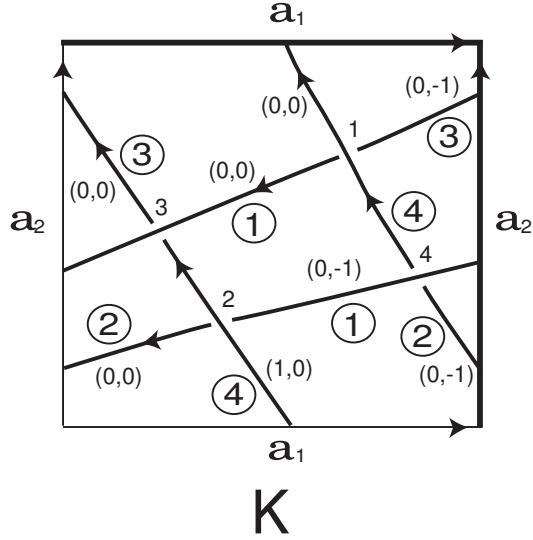}
\caption{Labels of arcs for the diagram in $T^{2}$} \label{ris13}
\end{figure}

Let $$G =\mathbb{Z}[t^{\pm1}, q, p^{\pm1}, x_1^{\pm1}, x_2^{\pm1},...,x_{2g}^{\pm1}]/ (q(p-t)=0,q^2=(1-t)(1-p))$$ be the quotient ring, $g$ being the number of handles.
Let us construct an $n\times{n}$-matrix $M(K)$ with elements in $G$ as follows. With each arc, we associate a column, and a row corresponds to each crossing. If some $j$-th arc is not incident to an $i$-th crossing, then we set $M_{ij}$ to be $0$. In the case, when only one subarc of the $j$-th arc is incident to the $i$-th crossing, then the element $M_{ij}$ shall be equal to the monomial $x_1^{\alpha_1}x_2^{\alpha_2}\cdot...\cdot x_{2g}^{\alpha_{2g}}$ multiplied by a local label, i.e. by one of the monomial $-1,1-t,t$ in even crossings and $-1,p,q$ in odd crossings, $(\alpha_1,\alpha_2,...,\alpha_{2g})$ is the label corresponding to the subarc  in this crossing (Fig. \ref{ris2}, Fig. \ref{ris1}). If there are several incident subarcs, then the corresponding element of the matrix is equal to the sum of such monomials for all incident subarcs.

We obtain the matrix $M(K)$, dependent on diagram $K$. Let $s(K)$ be $\mathrm{det}(M(K))$, where $s(K)\in G$.

\begin{theorem}
If two diagrams $K$ and $K'$ of knots in $S_g\times I$ are equivalent, then $s(K)=\pm s(K')\cdot t^{\alpha}p^{\beta}q^{\gamma}$ for some integers $\alpha,\beta,\gamma$.
\end{theorem}

\section{The Proof of the Theorem $1$}

Note that the crossing numeration change means that we switch columns
which yields an overall minus sign of the determinant. So we can enumerate the crossings arbitrarily.

Within this paragraph by invariance we shall always mean that the polynomial $s$ is invariant up to multiplication by $t^{\alpha}p^{\beta}q^{\gamma}$ for some integers $\alpha,\beta,\gamma$ unless specified otherwise.

Let us consider a following operation. Let $K$ be a diagram of a knot in $S_{g}\times I$, $l$ being some arc. We divide $l$ into two arcs $1$ and $2$ by adding a vertex $v$. The vertex $v$ is incident only to the arcs $1$ and $2$. All crossings which are incident to the arc $l$ divide into the crossings which incident to the arc $1$ and the crossings which incident to the arc $2$. Let us consider a matrix $M'(K)$ for obtained diagram. The row corresponding to the vertex $v$ has two non-zero elements in the $1$-st and $2$-nd columns (Fig. \ref{versh}). Note that the column corresponding to the arc $l$ subdivide into two columns such that their sum is equal the column corresponding to the $l$. Thus the determinants of $M(K)$ and $M'(K)$ are equal up to a multiplication by $t^\alpha$, $\alpha \in \mathbb{Z}$. Therefore the polynomial $s$ is the invariant under such dividing of the arc of the diagram $K$.

\begin{figure}
\centering\includegraphics[width=70pt]{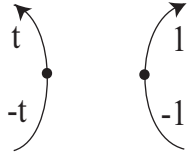}
\caption{Local labels under subdivision of an arc}
\label{versh}
\end{figure}

Since we have fixed sides of $P_{g}$ and represent diagrams on a $4g$-gon, we have to check the invariance not only under Reidemeister moves, but also under moves when a crossing pass through sides. Let us first prove the invariance of $s$ under moves when the crossing pass through sides in the case when the crossing $1$ is even (Fig. \ref{crosspass}). We divide an arc of diagrams $K$ and $K'$ into two arcs as shown in Fig. \ref{ris16} (see \cite{VO}); we assign the number $2$ to the new vertices.

\begin{remark}
The vertex $2$ in Fig. \ref{ris16} is even, because on chord diagrams this vertex does not generate chord, i.e. it does not intersect other chords.
\end{remark}

\begin{figure}
\centering\includegraphics[width=300pt]{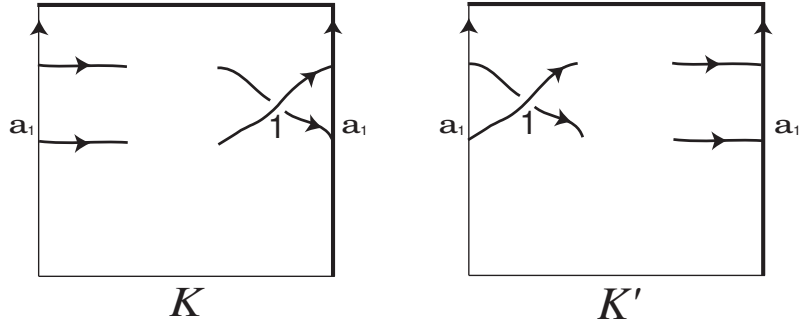}
\caption{A crossing passes through a side (first case)}
\label{crosspass}
\end{figure}

\begin{figure}
\centering\includegraphics[width=300pt]{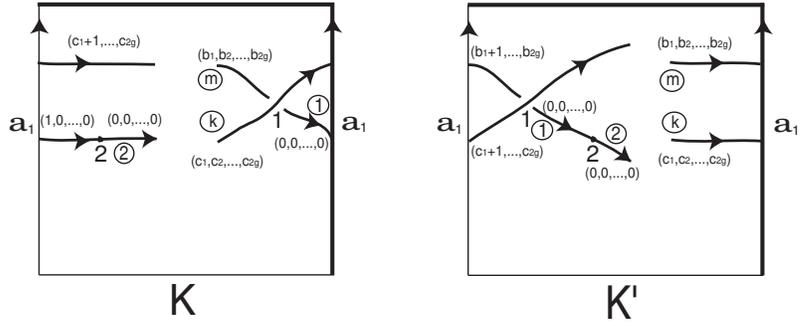}
\caption{The local labels after adding the vertex}
\label{ris16}
\end{figure}

As seen from Fig. \ref{ris16}, all rows except for the first two ones are identical. Let us write the corresponding matrices for the diagrams $K$ and $K'$:

\begin{equation}
M'(K)=\left(
\begin{array}{ccccccc}
 t& 0 & 0\dots 0 & -x_1^{b_1}\cdot...\cdot x_{2g}^{b_{2g}} & 0\dots 0 & (1-t)x_1^{c_1}\cdot...\cdot x_{2g}^{c_{2g}} & 0\dots 0 \cr -x_1 & 1 & 0\dots 0 & 0 & 0\dots 0 & 0 & 0\dots 0 \cr 0& * & * & * & * & * & *
\end{array}\right);
\label{1}
\end{equation}
\begin{equation}
M'(K')=\left(
\begin{array}{ccccccc}
 t& 0 & 0\dots 0 & -x_1^{b_1+1}\cdot...\cdot x_{2g}^{b_{2g}} & 0\dots 0 & (1-t)x_1^{c_1+1}\cdot...\cdot x_{2g}^{c_{2g}} & 0\dots 0 \cr -1 & 1 & 0\dots 0 & 0 & 0\dots 0 & 0 & 0\dots 0 \cr 0& * & * & * & * & * & *
\end{array}\right).
\label{2}
\end{equation}

Note, that all elements in the first column of both matrices are equal to $0$ except for the elements $M_{11}$ and $M_{21}$.
In (\ref{1}), let us multiply the first column by $x_1^{-1}$ and multiply the first row by $x_1$. As a result, we obtain the matrix which coincides with (\ref{2}).

Let us consider the case shown in Fig. \ref{ris17}, where the crossing $1$ is even. Then the corresponding matrices of the diagrams shall have the form:

\begin{figure}
\centering\includegraphics[width=300pt]{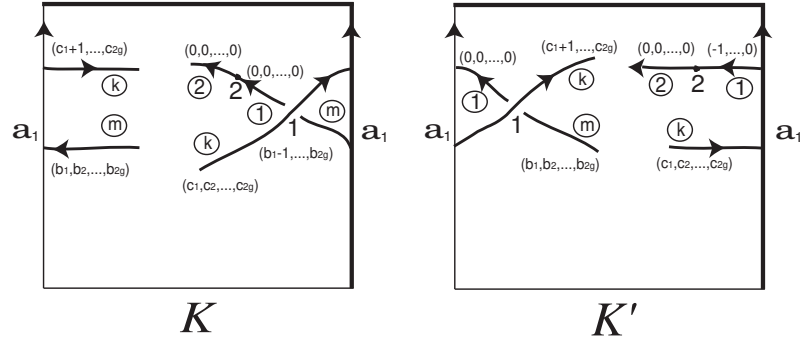}
\caption{A crossing passes through a side (second case)}
\label{ris17}
\end{figure}

\begin{equation}
M'(K)=\left(
\begin{array}{ccccccc}
  -1& 0 & 0\dots 0 & tx_1^{b_1-1}\cdot...\cdot x_{2g}^{b_{2g}} & 0\dots 0 & (1-t)x_1^{c_1}\cdot...\cdot x_{2g}^{c_{2g}} & 0\dots 0 \cr -1 & 1 & 0\dots 0 & 0 & 0\dots 0 & 0 & 0\dots 0 \cr 0 & * & * & * & * & * & *
\end{array}\right);
\label{3}
\end{equation}
\begin{equation}
M'(K')=\left(
\begin{array}{ccccccc}
 -1& 0 & 0\dots 0 & tx_1^{b_1}\cdot...\cdot x_{2g}^{b_{2g}} & 0\dots 0 & (1-t)x_1^{c_1+1}\cdot...\cdot x_{2g}^{c_{2g}} & 0\dots 0 \cr -x_1^{-1} & 1 & 0\dots 0 & 0 & 0\dots 0 & 0 & 0\dots 0 \cr 0 & * & * & * & * & * & *
\end{array}\right).
\label{4}
\end{equation}

Note, that in the first column in both matrices all elements are equal to $0$ except for the elements $M_{11}$ and $M_{21}$ just as in the previous case.
In the matrix (\ref{4}), let us multiply the first column by $x_1$ and the first row by $x_1^{-1}$. 	
Finally, we obtain the matrix identical to the matrix (\ref{3}).

Similarly we can prove the invariance in other cases of the move when a crossing pass through sides. Thus the polynomial $s$ is invariant under moves, at which one of the crossings passes through a side.

Now, to prove the theorem we verify the invariance under Reidemeister moves.

As we know (see for example \cite{Oht}), to establish the equivalence of two diagrams it is sufficient to use only one version of the $3$-rd Reidemeister move and all versions of the $1$-st and $2$-nd Reidemeister moves.

Let us first prove the invariance under the $1$-st Reidemeister move.
Suppose $K'$ is obtained from a diagram $K$ by the addition of a loop, see Fig. \ref{ris8}. As we know from the parity axiomatics, the crossing of the diagram $K'$ in the $1$-st move is even. We split the arc of the diagram $K$ into two arcs. Let us enumerate the crossings of the diagrams $K$ and $K'$ as follows: for the first crossing we take the crossing obtained
by the addition of the loop in the $K'$ and the new vertex in the $K$ and the numbers of the remaining crossings of $K$ and $K'$ accordingly identical, see Fig. \ref{ris8}.

\begin{figure}
\centering\includegraphics[width=300pt]{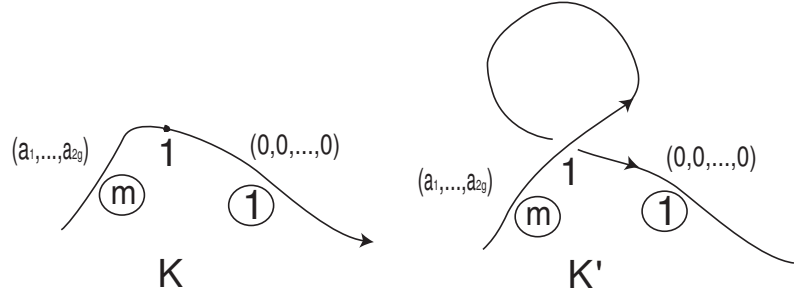} \caption{The first Reidemeister move} \label{ris8}
\end{figure}

Thus the rows of the matrices of the corresponding diagrams $K$ and $K'$ shall be identical everywhere except for the $1$-st row. If we write the first rows of these matrices, then we see that the determinants are equal up to a multiplication by $t^\alpha$, $\alpha\in \mathbb{Z}$. Indeed,

$$M'(K)=\left(
\begin{array}{cccc}
1 & 0\dots 0 & -x_1^{a_1}\cdot...\cdot x_{2g}^{a_{2g}} & 0 \dots 0\cr * & * & * & *
\end{array}\right);$$

$$M(K')=\left(
\begin{array}{cccc}
t & 0\dots 0 & -tx_1^{a_1}\cdot...\cdot x_{2g}^{a_{2g}} & 0 \dots 0\cr * & * & * & *
\end{array}\right).$$

Similarly we can prove the invariance under three other versions of the $1$-st Reidemeister move.

Let us now prove the invariance of the polynomial $s$ under the $2$-nd Reidemeister move. We first consider its co-oriented version. Suppose the diagram $K'$ is obtained from the diagram $K$ by an addition of two crossings as shown in Fig. \ref{ris9}.

\begin{figure}
\centering\includegraphics[width=200pt]{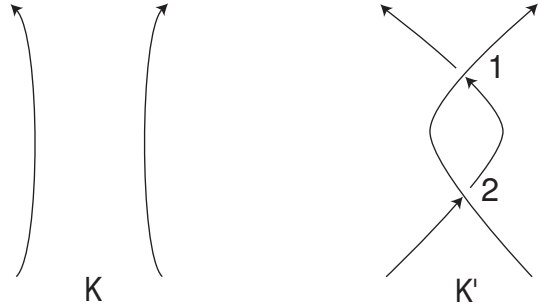} \caption{The second Reidemeister move} \label{ris9}
\end{figure}

According to the parity axiomatics both crossings in the $2$-nd Reidemeister move have the same parity. Hence we obtain two subcases in the case of the co-oriented Reidemeister move. We shall consider the crossings $1$ and $2$ are even in the $1$-st case and the crossings $1$ and $2$ are odd in the $2$-nd case.

Let us consider the first subcase. We divide the arc of the diagram $K$ by two vertices into three arcs and change the enumerations of the crossings of the diagrams $K$ and $K'$ as follows. For the $1$-st and $2$-nd crossings in the diagram $K'$ we take the crossings participating in the second move; in the diagram $K$ we take the new vertices and assign to the crossing with number $t$ of the diagram $K$
the number $t+2$ on the diagrams $K'$ and $K$ (Fig. \ref{ris4}).

\begin{figure}
\centering\includegraphics[width=300pt]{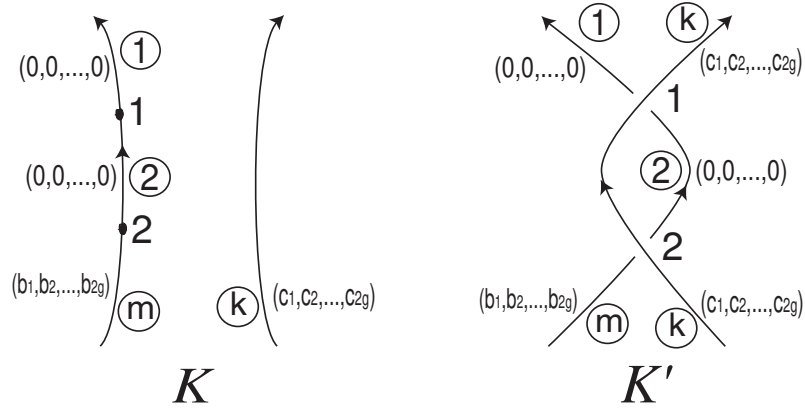}
\caption{Subdivision of an arc and construction of the local labels} \label{ris4}
\end{figure}

Note that the rows of the matrices are identical everywhere except for the $1$-st and $2$-nd rows. Let us write the matrices for the diagrams $K$ and $K'$:

\begin{equation}
M'(K)=\left(
\begin{array}{cccccc}
t & -t & 0\dots 0 & 0 & 0\dots 0 & 0\cr 0 & t & 0\dots 0 & 0 & 0\dots 0 & -tx_1^{b_1}\cdot...\cdot x_{2g}^{b_{2g}} \cr * & 0 & * & * & * & *
\end{array}\right);
\label{5}
\end{equation}
\begin{equation}
M(K')=\left(
\begin{array}{cccccc}
 -1 & t & 0\dots 0 & (1-t)x_1^{c_1}\cdot...\cdot x_{2g}^{c_{2g}} & 0\dots 0 & 0\cr 0 & t & 0\dots 0 & (1-t)x_1^{c_1}\cdot...\cdot x_{2g}^{c_{2g}} & 0\dots 0 & -x_1^{b_1}\cdot...\cdot x_{2g}^{b_{2g}} \cr * & 0 & * & * & * & *
\end{array}\right).
\label{6}
\end{equation}

We see that in the $2$-nd column in both matrices all elements are equal to $0$ except for $M_{12}$ and $M_{22}$.
In the matrix (\ref{5}) we add the $1$-st row to the $2$-nd row, in the matrix (\ref{6}) we multiply the $1$-st row by $-1$ and add to the $2$-nd row. We get:

\begin{equation}
\left(
\begin{array}{cccccc}
t & -t & 0\dots 0 & 0 & 0\dots 0 & 0\cr t & 0 & 0\dots 0 & 0 & 0\dots 0 & -tx_1^{b_1}\cdot...\cdot x_{2g}^{b_{2g}} \cr * & 0 & * & * & * & *
\end{array}\right);
\label{7}
\end{equation}
\begin{equation}
\left(
\begin{array}{cccccc}
 -1 & t & 0\dots 0 & (1-t)x_1^{c_1}\cdot...\cdot x_{2g}^{c_{2g}} & 0\dots 0 & 0\cr 1 & 0 & 0\dots 0 & 0 & 0\dots 0 & -x_1^{b_1}\cdot...\cdot x_{2g}^{b_{2g}} \cr * & 0 & * & * & * & *
\end{array}\right).
\label{8}
\end{equation}

Expanding the determinants of the matrices (\ref{7}) and (\ref{8}) with respect to the second column, we see that the determinant of the matrix (\ref{7}) is equal to the determinant of the matrix (\ref{8}) multiplied by $-t$.

Let us consider now the second subcase, when both crossings $1$ and $2$ are odd (Fig. \ref{ris4}). We write the matrices for the diagrams $K$ and $K'$:

\begin{equation}
M'(K)=\left(
\begin{array}{cccccc}
t & -t & 0\dots 0 & 0 & 0\dots 0 & 0\cr 0 & t & 0\dots 0 & 0 & 0\dots 0 & -tx_1^{b_1}\cdot...\cdot x_{2g}^{b_{2g}} \cr * & 0 & * & * & * & *
\end{array}\right);
\label{9}
\end{equation}
\begin{equation}
M(K')=\left(
\begin{array}{cccccc}
 -1& p & 0\dots 0 & qx_1^{c_1}\cdot...\cdot x_{2g}^{c_{2g}} & 0\dots 0 & 0\cr 0 & p & 0\dots 0 & qx_1^{c_1}\cdot...\cdot x_{2g}^{c_{2g}} & 0\dots 0 & -x_1^{b_1}\cdot...\cdot x_{2g}^{b_{2g}} \cr * & 0 & * & * & * & *
\end{array}\right).
\label{10}
\end{equation}

We note that in the $2$-nd column in both matrices all elements are equal to $0$ except for $M_{12}$ and $M_{22}$ as in the previous case. In the matrix (\ref{9}) we add the $1$-st row to the $2$-nd row and in the matrix (\ref{10}) we multiply the $1$-st row by $-1$ and add to the $2$-nd row. We obtain:

\begin{equation}
\left(
\begin{array}{cccccc}
t & -t & 0\dots 0 & 0 & 0\dots 0 & 0\cr t & 0 & 0\dots 0 & 0 & 0\dots 0 & -tx_1^{b_1}\cdot...\cdot x_{2g}^{b_{2g}} \cr * & 0 & * & * & * & *
\end{array}\right);
\label{11}
\end{equation}
\begin{equation}
\left(
\begin{array}{cccccc}
 -1 & p & 0\dots 0 & qx_1^{c_1}\cdot...\cdot x_{2g}^{c_{2g}} & 0\dots 0 & 0\cr 1 & 0 & 0\dots 0 & 0 & 0\dots 0 & -x_1^{b_1}\cdot...\cdot x_{2g}^{b_{2g}} \cr * & 0 & * & * & * & *
\end{array}\right).
\label{12}
\end{equation}

Expanding the determinants of the matrices (\ref{11}) and (\ref{12}) with respect to the second column, we see that the determinants are equal up to multiplication by $\pm t^\alpha p^\beta$, where $\alpha,\beta \in \mathbb{Z}$.

Similarly we can prove the invariance in the two subcases of the co-oriented version of the $2$-nd Reidemeister move and four subcases of the oppositely oriented version.

Let us prove the invariance under the third Reidemeister move. We enumerate the crossings of diagrams $K$ and $K'$ so that the numbers of the crossings participating in the Reidemeister move have the numbers from $1$ to $3$, and the numbers of other crossings were so that the arcs participating in the third Reidemeister move had the numbers from $1$ to $6$, Fig. \ref{ris6}.

\begin{figure}
\centering\includegraphics[width=400pt]{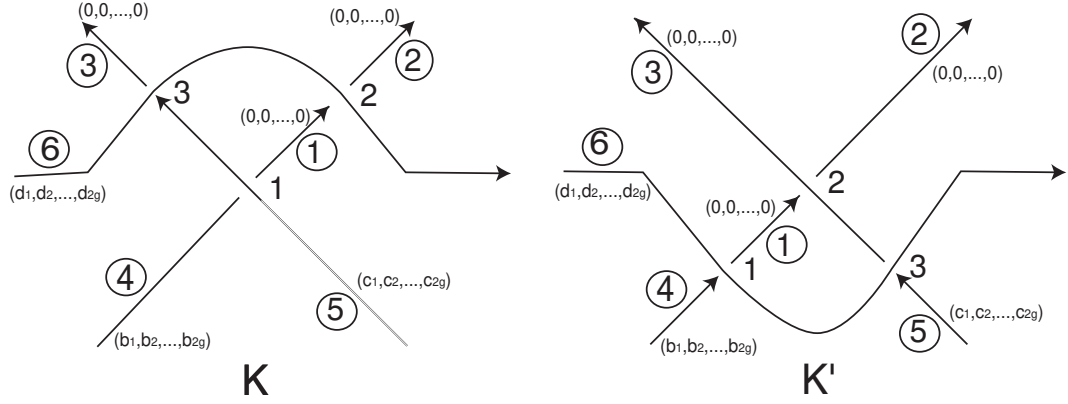} \caption{The third Reidemeister move} \label{ris6}
\end{figure}

As we know, the number of odd crossings among $1,2,3$ is even. Let us consider first the case when all three crossings are even. We write the matrices corresponding the diagrams $K$ and $K'$:

$$M(K)=\left(
\begin{array}{ccccccc}
t & 0 & 0 & -x_1^{b_1}\cdot...\cdot x_{2g}^{b_{2g}} & (1-t)x_1^{c_1}\cdot...\cdot x_{2g}^{c_{2g}} & 0 & 0\dots 0\cr t & -1 & 0 & 0 & 0 & (1-t)x_1^{d_1}\cdot...\cdot x_{2g}^{d_{2g}} & 0\dots 0 \cr 0 & 0 & -1 & 0 & tx_1^{c_1}\cdot...\cdot x_{2g}^{c_{2g}} & (1-t)x_1^{d_1}\cdot...\cdot x_{2g}^{d_{2g}} & 0 \dots 0 \cr
0 & * & * & * & * & * & *
\end{array}\right);$$

$$M(K')=\left(
\begin{array}{ccccccc}
 -1& 0 & 0 & tx_1^{b_1}\cdot...\cdot x_{2g}^{b_{2g}} & 0 & (1-t)x_1^{d_1}\cdot...\cdot x_{2g}^{d_{2g}} & 0\dots 0\cr -1 & t & 1-t & 0 & 0 & 0 & 0\dots 0 \cr 0 & 0 & -1 & 0 & tx_1^{c_1}\cdot...\cdot x_{2g}^{c_{2g}} & (1-t)x_1^{d_1}\cdot...\cdot x_{2g}^{d_{2g}} & 0 \dots 0 \cr
0 & * & * & * & * & * & *
\end{array}\right).$$

In both matrices in the $1$-st column we have $0$ starting from the $3$-rd row. All rows of these matrices are identical except for the $1$-st and $2$-nd rows.

In both matrices we multiply the $1$-st row by $-1$ and add it to the $2$-nd row:

\begin{equation}
\left(
\begin{array}{ccccccc}
 t& 0 & 0 & -x_1^{b_1}\cdot...\cdot x_{2g}^{b_{2g}} & (1-t)x_1^{c_1}\cdot...\cdot x_{2g}^{c_{2g}} & 0 & 0\dots 0\cr 0 & -1 & 0 & x_1^{b_1}\cdot...\cdot x_{2g}^{b_{2g}} & (t-1)x_1^{c_1}\cdot...\cdot x_{2g}^{c_{2g}} & (1-t)x_1^{d_1}\cdot...\cdot x_{2g}^{d_{2g}} & 0\dots 0 \cr 0 & 0 & -1 & 0 & tx_1^{c_1}\cdot...\cdot x_{2g}^{c_{2g}} & (1-t)x_1^{d_1}\cdot...\cdot x_{2g}^{d_{2g}} & 0 \dots 0 \cr
0 & * & * & * & * & * & *
\end{array}\right);
\label{13}
\end{equation}
\begin{equation}
\left(
\begin{array}{ccccccc}
 -1 & 0 & 0 & tx_1^{b_1}\cdot...\cdot x_{2g}^{b_{2g}} & 0 & (1-t)x_1^{d_1}\cdot...\cdot x_{2g}^{d_{2g}} & 0\dots 0\cr 0 & t & 1-t & -tx_1^{b_1}\cdot...\cdot x_{2g}^{b_{2g}} & 0 & (t-1)x_1^{d_1}\cdot...\cdot x_{2g}^{d_{2g}} & 0\dots 0 \cr 0 & 0 & -1 & 0 & tx_1^{c_1}\cdot...\cdot x_{2g}^{c_{2g}} & (1-t)x_1^{d_1}\cdot...\cdot x_{2g}^{d_{2g}} & 0 \dots 0 \cr
0 & * & * & * & * & * & *
\end{array}\right).
\label{14}
\end{equation}

Let us add the second column to the third column in (\ref{13}) and (\ref{14}). Finally we obtain the two matrices:

\begin{equation}
\left(
\begin{array}{ccccccc}
 t & 0 & 0 & -x_1^{b_1}\cdot...\cdot x_{2g}^{b_{2g}} & (1-t)x_1^{c_1}\cdot...\cdot x_{2g}^{c_{2g}} & 0 & 0\dots 0\cr 0 & -1 & -1 & x_1^{b_1}\cdot...\cdot x_{2g}^{b_{2g}} & (t-1)x_1^{c_1}\cdot...\cdot x_{2g}^{c_{2g}} & (1-t)x_1^{d_1}\cdot...\cdot x_{2g}^{d_{2g}} & 0\dots 0 \cr 0 & 0 & -1 & 0 & tx_1^{c_1}\cdot...\cdot x_{2g}^{c_{2g}} & (1-t)x_1^{d_1}\cdot...\cdot x_{2g}^{d_{2g}} & 0 \dots 0 \cr
0 & * & * & * & * & * & *
\end{array}\right);
\label{15}
\end{equation}
\begin{equation}
\left(
\begin{array}{ccccccc}
-1 & 0 & 0 & tx_1^{b_1}\cdot...\cdot x_{2g}^{b_{2g}} & 0 & (1-t)x_1^{d_1}\cdot...\cdot x_{2g}^{d_{2g}} & 0\dots 0\cr 0 & t & 1 & -tx_1^{b_1}\cdot...\cdot x_{2g}^{b_{2g}} & 0 & (t-1)x_1^{d_1}\cdot...\cdot x_{2g}^{d_{2g}} & 0\dots 0 \cr 0 & 0 & -1 & 0 & tx_1^{c_1}\cdot...\cdot x_{2g}^{c_{2g}} & (1-t)x_1^{d_1}\cdot...\cdot x_{2g}^{d_{2g}} & 0 \dots 0 \cr
0 & * & * & * & * & * & *
\end{array}\right).
\label{16}
\end{equation}

Now, we add the third row to the second row in (\ref{16}) and we multiply the third row by $-1$ and add it to the second row in (\ref{15}). We have:

\begin{equation}
\left(
\begin{array}{ccccccc}
t & 0 & 0 & -x_1^{b_1}\cdot...\cdot x_{2g}^{b_{2g}} & (1-t)x_1^{c_1}\cdot...\cdot x_{2g}^{c_{2g}} & 0 & 0\dots 0\cr 0 & -1 & 0 & x_1^{b_1}\cdot...\cdot x_{2g}^{b_{2g}} & -x_1^{c_1}\cdot...\cdot x_{2g}^{c_{2g}} & 0 & 0\dots 0 \cr 0 & 0 & -1 & 0 & tx_1^{c_1}\cdot...\cdot x_{2g}^{c_{2g}} & (1-t)x_1^{d_1}\cdot...\cdot x_{2g}^{d_{2g}} & 0 \dots 0 \cr
0 & * & * & * & * & * & *
\end{array}\right);
\label{17}
\end{equation}
\begin{equation}
\left(
\begin{array}{ccccccc}
-1 & 0 & 0 & tx_1^{b_1}\cdot...\cdot x_{2g}^{b_{2g}} & 0 & (1-t)x_1^{d_1}\cdot...\cdot x_{2g}^{d_{2g}} & 0\dots 0\cr 0 & t & 0 & -tx_1^{b_1}\cdot...\cdot x_{2g}^{b_{2g}} & tx_1^{c_1}\cdot...\cdot x_{2g}^{c_{2g}} & 0 & 0\dots 0 \cr 0 & 0 & -1 & 0 & tx_1^{c_1}\cdot...\cdot x_{2g}^{c_{2g}} & (1-t)x_1^{d_1}\cdot...\cdot x_{2g}^{d_{2g}} & 0 \dots 0 \cr
0 & * & * & * & * & * & *
\end{array}\right).
\label{18}
\end{equation}

Expanding the determinants of the matrices (\ref{17}) and (\ref{18}) with respect to the first column, we see that the determinants are equal.

Let us consider now the case, when the crossings $1$ and $3$ of the diagram $K$ are odd and the crossing $2$ is even (Fig.\ref{ris6}). 	
According to the parity axiomatics in the diagram $K'$ the crossings $1$ and $2$ are odd and the crossing $3$ is even. We write the matrices for the diagrams $K$ and $K'$:

$$M(K)=\left(
\begin{array}{ccccccc}
p & 0 & 0 & -x_1^{b_1}\cdot...\cdot x_{2g}^{b_{2g}} & qx_1^{c_1}\cdot...\cdot x_{2g}^{c_{2g}} & 0 & 0 \dots 0\cr t & -1 & 0 & 0 & 0 & (1-t)x_1^{d_1}\cdot...\cdot x_{2g}^{d_{2g}} & 0\dots 0 \cr 0 & 0 & -1 & 0 & px_1^{c_1}\cdot...\cdot x_{2g}^{c_{2g}} & qx_1^{d_1}\cdot...\cdot x_{2g}^{d_{2g}} & 0\dots 0 \cr
0 & * & * & * & * & * & *
\end{array}\right);$$

$$M(K')=\left(
\begin{array}{ccccccc}
-1 & 0 & 0 & tx_1^{b_1}\cdot...\cdot x_{2g}^{b_{2g}} & 0 & (1-t)x_1^{d_1}\cdot...\cdot x_{2g}^{d_{2g}} & 0\dots 0\cr -1 & p & q & 0 & 0 & 0 & 0\dots 0 \cr 0 & 0 & -1 & 0 & px_1^{c_1}\cdot...\cdot x_{2g}^{c_{2g}} & qx_1^{d_1}\cdot...\cdot x_{2g}^{d_{2g}} & 0 \dots 0 \cr
0 & * & * & * & * & * & *
\end{array}\right).$$

The first column of both matrices has $0$ entries in all positions starting from the third. All rows of these matrices are identical except for the $1$-st and $2$-nd rows.

In the first matrix we multiply the $1$-st row by $-\frac{t}{p}$ and add it to the second row; in the $2$-nd matrix we multiply the $3$-rd row by $q$ and add it to the second row. We obtain:

\begin{equation}\left(
\begin{array}{ccccccc}
p & 0 & 0 & -x_1^{b_1}\cdot...\cdot x_{2g}^{b_{2g}} & qx_1^{c_1}\cdot...\cdot x_{2g}^{c_{2g}} & 0 & 0 \dots 0\cr 0 & -1 & 0 & \frac{t}{p}x_1^{b_1}\cdot...\cdot x_{2g}^{b_{2g}} & -\frac{tq}{p}x_1^{c_1}\cdot...\cdot x_{2g}^{c_{2g}} & (1-t)x_1^{d_1}\cdot...\cdot x_{2g}^{d_{2g}} & 0\dots 0 \cr 0 & 0 & -1 & 0 & px_1^{c_1}\cdot...\cdot x_{2g}^{c_{2g}} & qx_1^{d_1}\cdot...\cdot x_{2g}^{d_{2g}} & 0\dots 0 \cr
0 & * & * & * & * & * & *
\end{array}\right);
\label{19}
\end{equation}

\begin{equation}\left(
\begin{array}{ccccccc}
-1 & 0 & 0 & tx_1^{b_1}\cdot...\cdot x_{2g}^{b_{2g}} & 0 & (1-t)x_1^{d_1}\cdot...\cdot x_{2g}^{d_{2g}} & 0\dots 0\cr -1 & p & 0 & 0 & qpx_1^{c_1}\cdot...\cdot x_{2g}^{c_{2g}} & q^2x_1^{d_1}\cdot...\cdot x_{2g}^{d_{2g}} & 0\dots 0 \cr 0 & 0 & -1 & 0 & px_1^{c_1}\cdot...\cdot x_{2g}^{c_{2g}} & qx_1^{d_1}\cdot...\cdot x_{2g}^{d_{2g}} & 0 \dots 0 \cr
0 & * & * & * & * & * & *
\end{array}\right).
\label{20}
\end{equation}

In (\ref{19}), we multiply the $1$-st column by $\frac{1}{p}$ and the $2$-nd row by $p$:

\begin{equation}\left(
\begin{array}{ccccccc}
1 & 0 & 0 & -x_1^{b_1}\cdot...\cdot x_{2g}^{b_{2g}} & qx_1^{c_1}\cdot...\cdot x_{2g}^{c_{2g}} & 0 & 0 \dots 0\cr 0 & -p & 0 & tx_1^{b_1}\cdot...\cdot x_{2g}^{b_{2g}} & -tqx_1^{c_1}\cdot...\cdot x_{2g}^{c_{2g}} & p(1-t)x_1^{d_1}\cdot...\cdot x_{2g}^{d_{2g}} & 0\dots 0 \cr 0 & 0 & -1 & 0 & px_1^{c_1}\cdot...\cdot x_{2g}^{c_{2g}} & qx_1^{d_1}\cdot...\cdot x_{2g}^{d_{2g}} & 0\dots 0 \cr
0 & * & * & * & * & * & *
\end{array}\right).
\label{21}
\end{equation}

In (\ref{20}), we multiply the $1$-st row by $-1$ and add it to the $2$-nd row:

\begin{equation}\left(
\begin{array}{ccccccc}
-1 & 0 & 0 & tx_1^{b_1}\cdot...\cdot x_{2g}^{b_{2g}} & 0 & (1-t)x_1^{d_1}\cdot...\cdot x_{2g}^{d_{2g}} & 0\dots 0\cr 0 & p & 0 & -tx_1^{b_1}\cdot...\cdot x_{2g}^{b_{2g}} & qpx_1^{c_1}\cdot...\cdot x_{2g}^{c_{2g}} & (q^2+t-1)x_1^{d_1}\cdot...\cdot x_{2g}^{d_{2g}} & 0\dots 0 \cr 0 & 0 & -1 & 0 & px_1^{c_1}\cdot...\cdot x_{2g}^{c_{2g}} & qx_1^{d_1}\cdot...\cdot x_{2g}^{d_{2g}} & 0 \dots 0 \cr
0 & * & * & * & * & * & *
\end{array}\right).
\label{22}
\end{equation}

Since the relations $q(p-t)=0$, $q^2=(1-t)(1-p)$ hold in the quotient ring $G$, then the matrix (\ref{22}) shall take the form:

\begin{equation}\left(
\begin{array}{ccccccc}
-1 & 0 & 0 & tx_1^{b_1}\cdot...\cdot x_{2g}^{b_{2g}} & 0 & (1-t)x_1^{d_1}\cdot...\cdot x_{2g}^{d_{2g}} & 0\dots 0\cr 0 & p & 0 & -tx_1^{b_1}\cdot...\cdot x_{2g}^{b_{2g}} & tqx_1^{c_1}\cdot...\cdot x_{2g}^{c_{2g}} & -p(1-t)x_1^{d_1}\cdot...\cdot x_{2g}^{d_{2g}} & 0\dots 0 \cr 0 & 0 & -1 & 0 & px_1^{c_1}\cdot...\cdot x_{2g}^{c_{2g}} & qx_1^{d_1}\cdot...\cdot x_{2g}^{d_{2g}} & 0 \dots 0 \cr
0 & * & * & * & * & * & *
\end{array}\right).
\label{23}
\end{equation}

Expanding the determinants of the matrices (\ref{21}) and (\ref{23}) with respect to the first column we see that the determinants are equal.

Arguing similarly for the other subcases of the third Reidemeister move we obtain that the matrices corresponding to the diagrams $K$ and $K'$ are equal.

\begin{example} Let us consider the diagrams $1.12$ and $\overline{1.13}$ of the knots in the torus $T^2$ from \cite{GMV1} (Fig. \ref{ris1.121}). We define the parity of the crossings of these diagrams (Fig. \ref{gaus1}). We see that all crossings of both diagrams are even. Let us assign the labels on the arcs of the given diagrams (Fig. \ref{ris1.122}).

\begin{figure}
\centering\includegraphics[width=400pt]{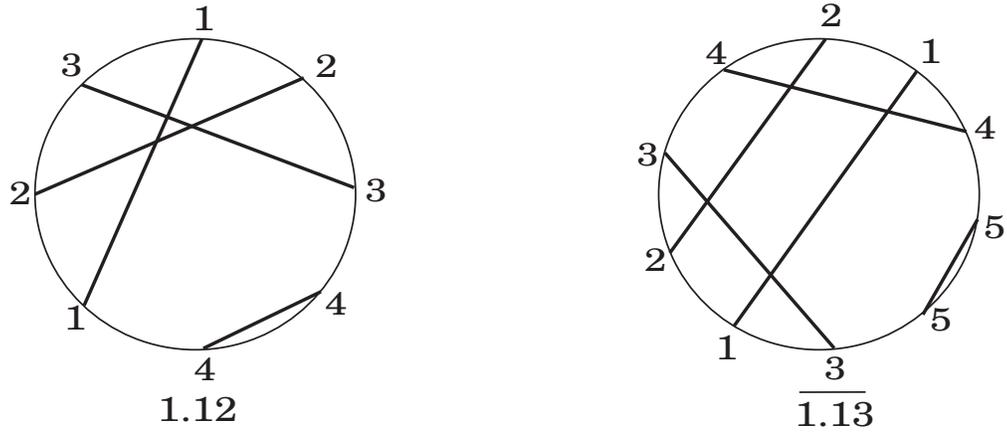} \caption{Chord diagrams of the knots $1.12$ and $\overline{1.13}$} \label{gaus1}
\end{figure}

\begin{figure}
\centering\includegraphics[width=400pt]{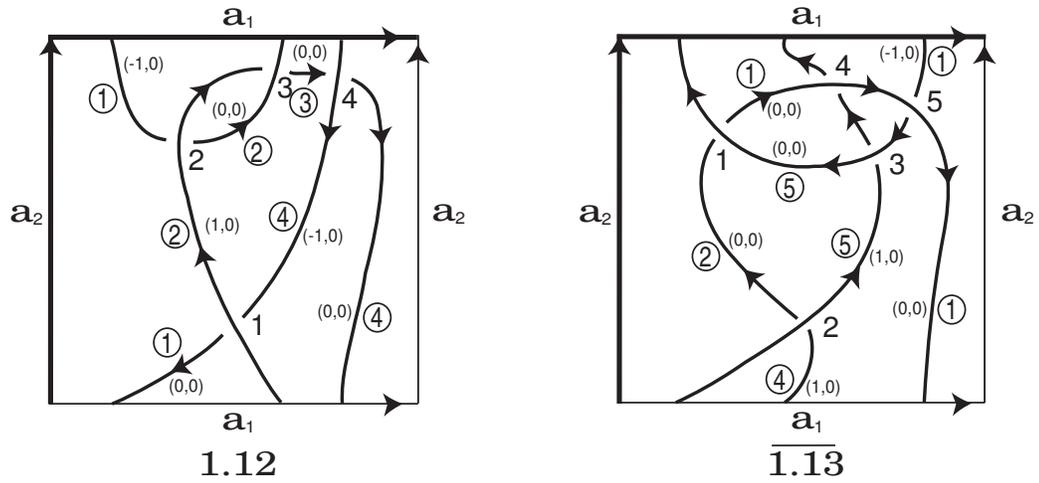} \caption{Diagrams with labels} \label{ris1.122}
\end{figure}

We write the matrices $M_{1.12}$ and $M_{\overline{1.13}}$ corresponding to the knot diagrams $1.12$ and $\overline{1.13}$.

$$
M_{1.12}=\left(
\begin{array}{cccc}
-1 & (1-t)x_1 & 0 & tx_1^{-1}
\cr -x_1^{-1} & t+(1-t)x_1 & 0 & 0
\cr 0 & 1-t-x_1 & t & 0
\cr 0 & 0 & t & (1-t)x_1^{-1}-1
\end{array}\right);
$$
$$
M_{\overline{1.13}}=\left(
\begin{array}{ccccc}
t & -1 & 0 & 0 & 1-t
\cr 0 & -1 & 0 & tx_1 & (1-t)x_1
\cr 0 & 0 & t & 0 & 1-t-x_1
\cr 1-t & 0 & t & -1 & 0
\cr 1-t-x_1^{-1} & 0 & 0 & 0 & t
\end{array}\right).
$$

We calculate the determinants of the obtained matrices:

$s(1.12)=\mathrm{det}(M_{1.12})=-2t+4t^2-t^3+\frac{t^2}{x_1^2}-\frac{t^3}{x_1^2}+\frac{t}{x_1}-\frac{4 t^2}{x_1}+\frac{2t^3}{x_1}+tx_1-t^2x_1;$

$s(\overline{1.13})=\mathrm{det}(M_{\overline{1.13}})=-2t+4t^2-t^3+\frac{t}{x_1}-\frac{t^2}{x_1}+tx_1-4t^2x_1+2t^3x_1+t^2x_1^2-t^3x_1^2.$

Thus, the constructed polynomials prove the non-equivalence of the knot diagrams $1.12$ and $\overline{1.13}$ in the torus $T^2$.

We note the invariant proves the non-equivalence of the knot diagrams $1.13$ and $\overline{1.12}$ too \cite{GMV1}.

\end{example}

\section{The parity Hierarchy and the Construction of an Invariant Module}

Let us start with definitions.

\begin{definition}
\emph{A virtual knot diagram} is a planar graph of valency four endowed with the following structure: each vertex either has an over/under crossing structure or is marked by \emph{a virtual crossing} as shown in Fig. \ref{virt1}. Virtual knots are equivalence classes of virtual knot diagrams modulo generalized Reidemeister moves: classical Reidemeister moves which refer to classical crossings only and the detour move. The latter represents
the following: A branch of a knot diagram containing several consecutive virtual crossings but not containing classical crossings can be transformed
into any other branch with the same endpoints; new intersections and self-intersections are marked as virtual crossings (Fig. \ref{detour}).
\end{definition}

\begin{figure}
\centering\includegraphics[width=50pt]{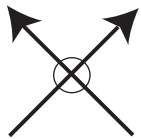} \caption{Virtual crossing} \label{virt1}
\end{figure}

\begin{figure}
\centering\includegraphics[width=300pt]{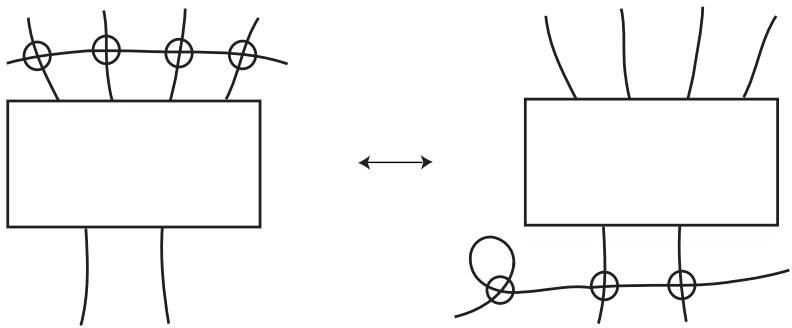} \caption{Detour move} \label{detour}
\end{figure}

All crossings except virtual ones are said to be classical \cite{MyBook}. Virtual knots are knots in thickened $2$-surfaces considered up to isotopy and stabilizations/destabilizations. In principle, the invariant we are going to construct can be constructed for knots in a concrete thickened surface but we want to restrict ourselves to virtual knots, so we shall not use variables $x_1,..., x_{2g}$, corresponding to handles.
The aim of the present section is to construct an invariant of virtual knots. From the module constructed in the present section one can extract a Laurent polynomial invariant of virtual knots, which shall be done in Section $5$.
Both invariants, the module itself and the polynomial, can be easily generalized for the case of a fixed thickened surface and with some enhancement coming from new generators corresponding to meridians.
However, this is rather straightforward and we are not going to dwell into it.

 The main idea behind the result of the present section is to use the parity hierarchy. In Section $2$, we distinguished between two types of classical crossings, the even ones and the odd ones, and applied different relations to different types of crossings. Now, we are going to use the parity hierarchy first introduced by V.O.Manturov in \cite{IMN}.
The Gaussian parity discriminates between even and odd crossings. It turns out that there is a natural way for a further discrimination. Namely, let $K$ be a virtual knot diagram, and let $f(K)$ be the image of $K$ under the map $f$ which maps odd crossings to virtual crossings.
Then all crossings of $f(K)$ correspond to even crossings of $K$. These crossings treated as crossings of $f(K)$ can be either even or odd.
Thus, we say that a crossing of $K$ is of type $0$ if it is odd,
otherwise we say that it is of type $1$ if the corresponding crossing of $f(K)$ is odd,
and if the corresponding crossing of $f(K)$ is even then we say that the initial crossing is of type $2$ \cite{H}. In each crossing we have relations, shown in Fig. \ref{virt} and \ref{klass}.

Note, that
1) the crossing obtained by adding a loop in the $1$-st move has type $2$;

2) in the $2$-nd move both crossings either have type $0$ or type $1$, or type $2$;

3) in the $3$-rd move we have the following cases:

a) all crossings $1,2,3$ have type $2$;

b) two crossings among  $1$, $2$, $3$ have type $0$ and the remaining one has type $1$;

c) two crossings among  $1$, $2$, $3$ have type $0$ and the remaining one has type $2$;

d) two crossings among  $1$, $2$, $3$ have type $1$ and the remaining one has type $2$.

\begin{figure}
\centering\includegraphics[width=200pt]{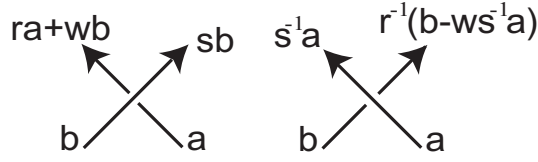} \caption{A crossing of type $0$} \label{virt}
\end{figure}

\begin{figure}
\centering\includegraphics[width=180pt]{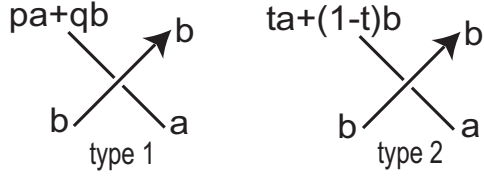} \caption{Crossings of types $1$ and $2$} \label{klass}
\end{figure}

Set $$R=\mathbb{Z}[t^{\pm1}, q, p^{\pm1}, s^{\pm1}, r^{\pm1}, w] / (q(p-t)=0, q^2=(1-t)(1-p), w(1-s)=0, w(t-r)=0,$$ $$w^2=(1-t)(1-rs), w(ps+q-1)=0, w(r+q-1)=0, w(p-r)=0, w^2=q(1-rs))$$ to be the quotient ring. Let us construct a module $N(K)$ over $R$. The generators are the short arcs of the knot diagram. At every crossing, we shall use the relations given in Fig. \ref{virt} and \ref{klass} according to the type of the crossing. Note, that only for crossings of type $0$ we pay attention to the orientation of the underpass;
thus we have two different rules for positive and negative crossings. For crossings of types $1$ and
$2$, we use only one relation as in the definition of the local labels for the polynomial $s$.

\begin{remark}
We can construct invariants coming from higher parity hierarchy in a similar way; this hierarchy shall lead to similar formulae, but for brevity we restrict ourselves to this parity of the third level.
\end{remark}

\begin{remark}
Unlike the previous sections, where we dealt with arcs of the diagram, now we shall use some "shortened" versions of arcs.
Namely, by a {\em short arc} we mean a branch of a knot diagram which goes from an underpass of a classical crossing (of any type) or an overpass of a classical crossing of type $0$ to the next underpass or overpass of type $0$. Thus, a short arc may contain virtual crossings and overpasses of types $1$ and $2$, but neither classical underpasses nor overpasses of type $0$. This is done for the following reason. At crossings of type $0$ we would like to introduce a module structure where both emanating arcs linearly depend on incoming arcs, however, the emanating part of the overcrossing arc shall not be equal to the incoming part of the same arc.
The number of generators (crossings) of the module to be constructed shall not be equal to the number of relations. However, in the very end of the paper, we shall construct a polynomial invariant coming from a simplification of this module. The idea behind this invariant is to treat crossings of type $0$ as virtual, and this turns out to be a partial case of what we are doing in the present section.
\end{remark}

\begin{theorem}
The module $N(K)$ is an invariant of virtual knots.
\end{theorem}

\begin{proof}

To prove the theorem we verify invariance under the Reidemeister moves.
The invariance under the $1$-st move is proved just as in the case of the invariance of the Alexander polynomial.
Let us prove the invariance under the $2$-nd move in the case shown in Fig. \ref{move2module1}. The crossings $1$ and $2$ have the same type. When the crossings $1$ and $2$ have type $2$, then in this case the invariance follows from the invariance of the Alexander polynomial, too.
Let us consider the case when the crossings $1$ and $2$ are of type $1$. We see in Fig. \ref{move2module1} that the expressions of the emanating short arcs in terms of incoming short arcs for the diagrams $K$ and $K'$ are equal. Similarly we can prove the invariance under the other three cases of the $2$-nd Reidemeister move when both crossings $1$ and $2$ are of type $1$.

\begin{figure}
\centering\includegraphics[width=200pt]{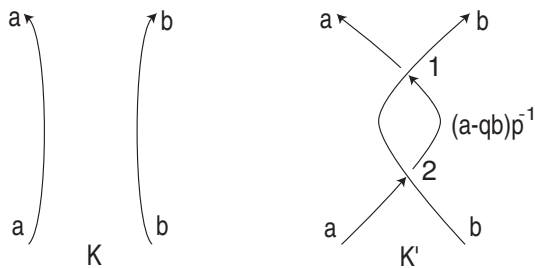} \caption{The case of the second Reidemeister move with crossings of type $1$} \label{move2module1}
\end{figure}

Now, let us prove the invariance in the case when the crossings $1$ and $2$ are of type $0$. We consider the case shown in Fig. \ref{move2module0}. Extending the emanating short arcs as linear combinations of incoming short arcs we see that the expressions of the emanating short arcs in terms of incoming short arcs for the diagrams $K$ and $K'$ are equal. Similarly we can prove the invariance under the other three cases of the $2$-nd Reidemeister move when the crossings $1$ and $2$ are of type $0$.

\begin{figure}
\centering\includegraphics[width=200pt]{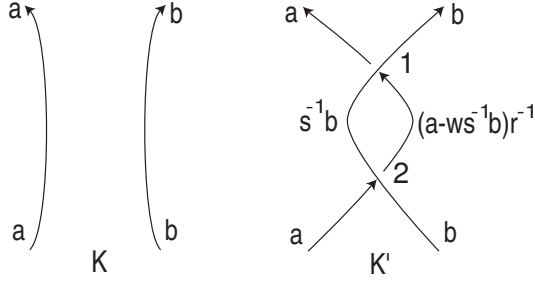} \caption{The case of the second Reidemeister move with crossings of type $0$} \label{move2module0}
\end{figure}

Let us prove the invariance under the third Reidemeister move. As already mentioned, to establish the equivalence of the two diagrams it is sufficient to use only one version of the $3$-rd Reidemeister move. Suppose that the diagram $K$ is obtained from the diagram $K'$ by the $3$-rd move (Fig. \ref{3move}).

\begin{figure}
\centering\includegraphics[width=250pt]{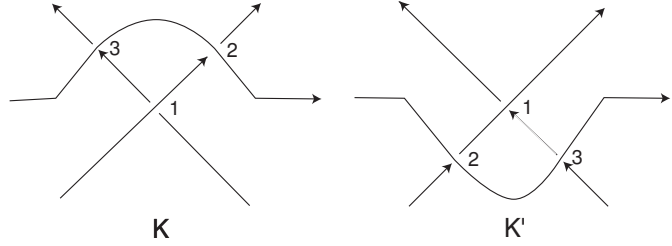} \caption{The third Reidemeister move} \label{3move}
\end{figure}

Note, that the case $3a$ follows from the invariance of the Alexander polynomial. Let us prove the invariance in the case $3b$, when in the diagrams $K$ and $K'$ the crossings $1$ and $3$ are of type $0$ and the crossing $2$ is of type $1$ (Fig. \ref{3move100}).

\begin{figure}
\centering\includegraphics[width=350pt]{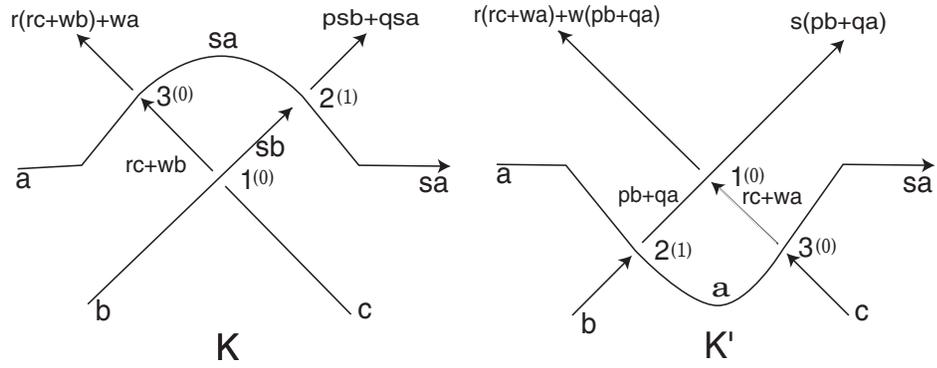} \caption{The case $3b$ of the third Reidemeister move} \label{3move100}
\end{figure}

Since the relations $w(r+q-1)=0, w(p-r)=0$ hold in the quotient ring $R$, then in this subcase the module $N$ is the invariant. Similarly we can prove the invariance under other two subcases of the case $3b$. Now, let us prove the invariance in the case $3c$. Let the crossings $1$ and $2$ be of type $0$ and the crossing $3$ be of type $2$. Then the expression of the emanating short arcs in terms of incoming short arcs for the diagrams $K$ and $K'$ shall be as shown in Fig. \ref{3move200}. Since the relations $w(1-s)=0, w^2=(1-t)(1-rs), w(t-r)=0$ hold in the quotient ring $R$, then in this subcase the module $N$ is invariant. Similarly we can prove the invariance under other two subcases of the $3$-rd case. The invariance in the case $3d$ follows from the invariance of the polynomial $s$.

\begin{figure}
\centering\includegraphics[width=350pt]{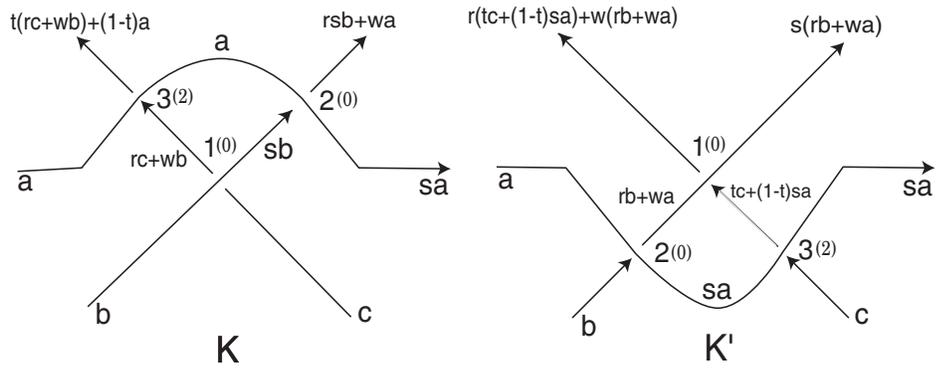} \caption{The case $3c$ of the third Reidemeister move} \label{3move200}
\end{figure}

Thus, the constructed module $N(K)$ over the quotient ring $R$ is invariant of virtual knots. \end{proof}

\section{Particular case of the module $N(K)$}

It turns out that the invariant module $N(K)$ can be simplified in order to obtain an invariant polynomial for virtual knots.

Indeed, let us construct a simplification of the module $N(K)$, to be denoted by $N'(K)$, which is constructed as follows. First, for the ground ring we take $$R'=\mathbb{Z}[t^{\pm1}, q, p^{\pm1}, s^{\pm1}]/(q(p-t)=0, q^2=(1-t)(1-p)).$$ Then, we take short arcs to be generators as in the case of the module $N$ and crossings to be relations. However, we make the following simplifications. In crossings of types $1$ and $2$ we have the relations as shown in Fig. \ref{klass}. The crossing of type $0$ we replace by the virtual crossing and have the relations as shown in Fig. \ref{virt2}. One can obtain this simplification of the module $N(K)$ if $r=s^{-1},w=0$ for the module $N(K)$.

Obvious

\begin{theorem}
The module $N'(K)$ is an invariant of virtual knots.
\end{theorem}

The new module $N'(K)$ allows one to construct a simple representation and from the invariance of $N'(K)$ follows the invariance of a polynomial $n'(K)$ in $R'$. The module $N(K)$ can be treated as follows: at every crossing of type $0$, the label of a short arc gets multiplied by $s$ or by $s^{-1}$. Thus, we can think of a virtual diagram $K$ as having two types of virtual crossings: those originally being virtual and those of type $0$ which "disappear" or "become virtual" after the projection map $f$. Nevertheless, the newborn crossings preserve a piece of information, namely, when passing through the crossing, we multiply the label by $s^{\pm 1}$. This can be easily packaged into a matrix with $n$ generators and $n$ relations, where $n$ is the number of crossings of types $1$ and $2$. Crossings of type $0$ are ignored, and different arcs approaching the same crossing of type $0$, have labels which differ by $s^{\pm 1}$. This leads one to the matrix which is constructed as follows.

Let $K$ be a diagram of the virtual knot with $n$ classical crossings. We construct an $n\times n$-matrix $N''(K)$ as follows. With each short arc, we associate a column, and a row corresponds to each crossing. If some $j$-th short arc is not incident to a $i$-th crossing, then we set $N''_{ij}$ to be $0$. In the case, when only one short arc of the $j$-th arc is incident to the $i$-th crossing, then the element $N''_{ij}$ shall be equal to one of the monomial $-1,1-t,t$ in crossings of type $2$ and $-1,p,q$ in crossings of type $1$ (Fig. \ref{klass}, Fig. \ref{virt2}). If there are several incident short arcs, then
corresponding element of the matrix is equal to the sum of such monomials for all incident short arcs.
We obtain the matrix $N''(K)$ dependent on diagram $K$. Let $n'(K)$ be $\mathrm{det}(N''(K))$, where $n'(K)\in R'$.

\begin{figure}
\centering\includegraphics[width=80pt]{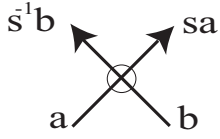} \caption{Relations in the virtual crossing} \label{virt2}
\end{figure}

\begin{theorem}
If two diagrams $K$ and $K'$ of virtual knots are equivalent, then $n'(K)=\pm n'(K')\cdot t^{\alpha}p^{\beta}q^{\gamma}$ for some integers $\alpha,\beta,\gamma$.
\end{theorem}

Principally, the proof of this theorem repeats the invariance proof of the module
because the defining relations for the module correspond to the rows of the matrix.
Thus, the equivalence of two defining sets of relations means that the corresponding
matrices have the same determinant up to multiplication by the invertible elements
of the ground ring. Certainly, all these modifications of matrices can be written
down explicitly.

\section*{Acknowledgments}

I express my gratitude V.O.Manturov for the formulation of the problem, useful remarks in preparation of this article and fruitful consultations.

\end{document}